\documentclass[pdflatex,11pt,a4paper]{article}

\usepackage{mathptmx} 
\DeclareMathAlphabet{\mathcal}{OMS}{cmsy}{m}{n}
\usepackage[top=25mm,bottom=35mm,left=20mm,right=20mm]{geometry} 

\usepackage{graphicx}
\usepackage{amssymb,amsmath,amsfonts,amsthm}
\usepackage{bm}
\usepackage{url}
\usepackage{doi}
\usepackage{tikz}
\usepackage{pgfplotstable}
\usepackage{mathtools}
\mathtoolsset{showonlyrefs=true}
\DeclarePairedDelimiter\paren{\lparen}{\rparen}

\newtheorem{theorem}{Theorem}
\newtheorem{lemma}{Lemma}
\newtheorem{proposition}{Proposition}

\newtheorem{remark}{Remark}

\makeatletter
\newcommand\Autoref[1]{\@first@ref#1,@}
\def\@throw@dot#1.#2@{#1}% discard everything after the dot
\def\@set@refname#1{%    % set \@refname to autoefname+s using \getrefbykeydefault
    \edef\@tmp{\getrefbykeydefault{#1}{anchor}{}}%
    \def\@refname{\@nameuse{\expandafter\@throw@dot\@tmp.@autorefname}s}%
}
\def\@first@ref#1,#2{%
  \ifx#2@\autoref{#1}\let\@nextref\@gobble% only one ref, revert to normal \autoref
  \else%
    \@set@refname{#1}%  set \@refname to autoref name
    \@refname~\ref{#1}% add autoefname and first reference
    \let\@nextref\@next@ref% push processing to \@next@ref
  \fi%
  \@nextref#2%
}
\def\@next@ref#1,#2{%
   \ifx#2@ and~\ref{#1}\let\@nextref\@gobble% at end: print and+\ref and stop
   \else, \ref{#1}% print  ,+\ref and continue
   \fi%
   \@nextref#2%
}
\makeatother

\newcommand{\calE}{\mathcal E}

\newcommand{\calI}{\mathcal I}

\newcommand{\calM}{\mathcal M}
\newcommand{\calS}{\mathcal S}

\newcommand{\rmd}{\mathrm d}
\newcommand{\rme}{\mathrm e}
\newcommand{\rmi}{\mathrm i}

\newcommand{\bbR}{\mathbb R}
\newcommand{\bbT}{\mathbb T}
\newcommand{\disT}{X_{\mathrm d}}
\newcommand{\bbZ}{\mathbb Z}
\newcommand{\pv}{{\mathrm {P.V.}}}
\newcommand{\Se}{S_{\mathrm{even}}}
\newcommand{\So}{S_{\mathrm{odd}}}

\newcommand{\bmd}{\bm d}
\newcommand{\bmu}{\bm u}
\newcommand{\bmv}{\bm v}

\newcommand{\px}{\partial _x }
\newcommand{\pt}{\partial _t }

\newcommand{\dx}{\delta _x }
\newcommand{\dt}{\delta _t }
\newcommand{\Dx}{\Delta x }
\newcommand{\Dt}{\Delta t }
\newcommand{\dht}{H_{\Dx} }
\newcommand{\dlt}{L_{\Dx} }
\newcommand{\dft}{F_{\Dx}}
\newcommand{\ddo}{D_{\Dx}}

\newcommand{\fracpar}[2]{{\frac{\partial #1}{\partial #2} } }
\newcommand{\fracdel}[2]{{\frac{\delta #1}{\delta #2} } }

\makeatletter
\def\cot{\mathop{\operator@font cot}\nolimits}
\def\sgn{\mathop{\operator@font sgn}\nolimits}
\def\diag{\mathop{\operator@font diag}\nolimits}
\makeatother

\begin{document}

\title{Structure-preserving integrators for the Benjamin-type equations}
\author{Kimiaki KINUGASA\thanks{Graduate School of Information Science and Technology, 
The University of Tokyo, 
\href{mailto:kimiaki_kinugasa@ipc.i.u-tokyo.ac.jp}{kimiaki{\_}kinugasa@ipc.i.u-tokyo.ac.jp}}, \
Yuto MIYATAKE\thanks{Graduate School of Engineering, Nagoya University, 
\href{mailto:miyatake@na.cse.nagoya-u.ac.jp}{miyatake@na.cse.nagoya-u.ac.jp}} \ and 
Takayasu MATSUO\thanks{Graduate School of Information Science and Technology, 
The University of Tokyo, 
\href{mailto:matsuo@mist.i.u-tokyo.ac.jp}{matsuo@mist.i.u-tokyo.ac.jp}} }
%\date{August, 2015}

\maketitle

\begin{abstract}
The numerical integration of the Benjamin and Benjamin--Ono equations are considered.
They are non-local partial differential equations involving the Hilbert transform,
and due to this, so far quite few structure-preserving integrators have been proposed.
In this paper, a new reformulation of the equations is stated,
and new structure-preserving discretizations are proposed based on it.
Numerical experiments confirm the effectiveness of the proposed integrators.
\end{abstract}

\section{Introduction}
\label{intro}
In this paper, we consider structure-preserving numerical integration
of the Benjamin-type equations.
The Benjamin equation is represented as
\begin{align}
u_t + \gamma u_x + \lambda uu_x -\alpha L u_{x} -\beta u_{xxx} = 0,
\label{Benjamin}
\end{align}
where $u=u(t,x)$, $x\in\bbR$, $t\geq 0$, $\alpha, \beta, \gamma, \lambda $ are real parameters,
the subscript $t$ (or $x$, respectively) denotes the differentiation
with respect to the time variable $t$ (or $x$).
The operator $L$ is defined as
$L= H \px$, where $H$ denotes the Hilbert transform
\begin{align}
H u(x) = \frac{1}{\pi}\pv \int_{-\infty}^\infty\frac{u(x-y)}{y}\,\rmd y.
\label{HT}
\end{align}
This equation is often considered on the torus $\bbT$, i.e. under 
the periodic boundary condition of length $l$.
In this case the Hilbert transform is defined by 
\begin{align}\label{htf2}
Hu(x) = \frac{1}{l} \pv \int_0^l \cot \paren*{\frac{\pi}{l}y} u(x-y)\,\rmd y,
\end{align}
or equivalently through its Fourier transform
\begin{align}
F(Hu)(k) = -\rmi \sgn (k) \cdot F(u) (k).
\end{align}
The Benjamin equation can formally be seen as a generalization of 
the KdV equation ($\alpha=\gamma=0$) and
the Benjamin--Ono equation ($\beta=\gamma=0$)~\cite{be67,on75}.

The Benjamin equation was first introduced in~\cite{be92}
as a governing equation which models unidirectional propagation of 
long internal waves of small amplitude at an interface of two incompressible fluids of different density.
Global well-posedness is proved for data in both $L^2(\bbR)$ and $L^2(\bbT)$~\cite{li99}.
Since the KdV and Benjamin--Ono equations possess solitary wave solutions of the form
$u(t,x) = \varphi (x-ct)$,
whether the Benjamin equation also has solitary waves has been a major subject.
This topic was initially raised by Benjamin~\cite{be92,be96},
and extensively studied after that (See, for example, \cite{abr99,ap99,cb98}).
Nevertheless, in contrast to the KdV and Benjamin--Ono equations,
explicit formulae are not known.
Since then, numerical studies have been done to approximate the solitary waves~\cite{abr99,ddm12,ddm15,kb00}.

The Benjamin equation has three invariants~\cite{be92}
\begin{align*}
\calM &= \int_{-\infty}^\infty u\,\rmd x, \\
\calI &= -\int_{-\infty}^\infty \frac12 u^2 \,\rmd x ,\\
\calE &= \int_{-\infty}^\infty \paren*{- \frac{\gamma}{2}u^2 - \frac{\lambda}{6}u^3 + \frac{\alpha}{2}uLu
-\frac{\beta}{2}u_x^2}\,\rmd x.
\end{align*}
They are also constant when the equation is considered on the torus $\bbT$.
Only the third invariant $\calE$ determines a Hamiltonian structure $u_t = \px \delta \calE / \delta u$.

For PDEs with geometric structures,
it is widely accepted that structure-preserving numerical methods often
yield better numerical solutions than general-purpose methods, especially over a long period of time,
and this topic has attracting much attention in the last two decades
(see e.g.~\cite{hlw06} for the temporal discretization
and~\cite{br06,fq10,fm11,lr04} for the spatial discretization).
For the KdV equation, several structure-preserving methods have been proposed;
for example,
invariants-preserving methods~\cite{cgm12,do11,fu99},
and symplectic and multi-symplectic methods~\cite{am04,am05,zq00}.
On the other hand,
for the Benjamin and Benjamin--Ono equations, 
only few structure-preserving methods have been known,
because the study on geometry of the Benjamin-type equations are less developed,
and the non-local operator ($H$ or $L$) needs to be discretized carefully.
We are only aware of a single exception: 
an $\calI$-preserving method for the Benjamin--Ono equation,
which was proposed by Thom{\'e}e--Vasudeva Murthy~\cite{tm98}.

Based on the above observation, in this paper 
we aim at proposing some new structure-preserving integrators for the Benjamin-type equations.
We like to note here that some invariant-preserving discretizations are rather obvious
in the following sense.
First, Thom{\'e}e--Vasudeva Murthy's approach mentioned above
can be readily applied to the Benjamin equation.
Second, it is straightforward to derive an $\calE$-preserving integrator
based on the Hamiltonian structure mentioned above
by utilizing the discrete variational derivative method~\cite{cgm12,fu99,fm11}
(provided an appropriate discretization of the Hilbert transform,
such as the one in Thom\'ee--Vasudeva Murthy, or in the present paper).
Thus in the present paper we like to try a different approach,
which is in some sense on top of the literature of the so-called multi-symplectic method.
To this end, we first have to find
a multi-symplectic formulation of the Benjamin type equations.
If it is found, it should tell us the local behavior of the Benjamin type equations;
unfortunately, however, it seems the Hilbert transform 
(which is a non-local operator) prohibits that, 
at least in a standard manner,
and we have to introduce some new ideas.

In this paper, we do this by extending the concept of the multi-symplecticity so that it can fit into 
the Benjamin-type equations.
Then we discretize the equations
by the Euler box and Preissmann box schemes.
Here arises another difficulty---%
the Preissmann box scheme is generally stabler than the Euler box scheme,
and thus is more preferable;
but as its price it has a disadvantage that it is not uniquely solvable
unless the number of the spatial grid points is odd
(this is caused by an averaging operator in front of the unknown variable).
This is troublesome in the present context, since in the literature,
the discretization of the Hilbert transform has been considered only with even number
of grid points.
In this study, we give a discretization of the Hilbert transform also for
the odd case, and more importantly, give its theoretical justification.

This paper is organized as follows.
In \autoref{sec2},
we briefly review the concept of multi-symplecticity
and some discretization methods.
In \autoref{sec3}, we propose new integrators for the Benjamin-type equations.
We extend the concept of the multi-symplecticity in \autoref{sec31},
discuss the discretization of the Hilbert transform and the operator $L$ in \autoref{sec32},
and derive integrators in \Autoref{sec33,sec34}.
In \autoref{sec4}, some numerical results are provided.
Finally, concluding remarks are given in \autoref{sec5}.

The following notation is used in this paper.
For the sake of numerical computation, 
we impose the periodic boundary condition.
The domain $[0,l]$ is discretized by uniform meshes with the space mesh size $\Dx = l/N$.
Numerical solutions are denoted by $u_n^i \approx u(i\Dt , n\Dx)$ where $\Dt$
is the time mesh size.
When we write only  the subscript or superscript, it means 
the associated semi-discretization.
We use the abbreviation $u_{n+1/2} = (u_n+u_{n+1})/2$
(a similar abbreviation is also used for the time index).
In order to treat the periodic boundary condition,
we restrict our consideration to an infinite long vector $\bmu=\{u_n\}_{n\in\bbZ}$
with the property $u_n = u_{n+N}$,
and denote the space to which  such periodic vectors belong by $\disT$
(in this paper, a bold type always belongs to $\disT$).
We use the standard difference operators that approximate $\px$ and $\pt$:
\begin{alignat}{3}
\dx^+ u_n &= \frac{u_{n+1}-u_n}{\Dx},
&\qquad
\dx^- u_n &= \frac{u_n-u_{n-1}}{\Dx},
&\qquad
\dx u_n &= \frac{u_{n+1}-u_{n-1}}{2\Dx}, \\
\dt^+ u^i &= \frac{u^{i+1}-u^i}{\Dt},
&\qquad
\dt^- u^i &= \frac{u^i-u^{i-1}}{\Dt},
&\qquad
\dt u^i &= \frac{u^{i+1}-u^{i-1}}{2\Dt}.
\end{alignat}

\section{Preliminaries}
\label{sec2}
In this section, we briefly review the concept of multi-symplecticity
and some discretization methods.
For more details on this topic,
we refer, for example, to the early references~\cite{br97,br01,mr03}.

\subsection{Multi-symplectic PDEs}
A partial differential equation $F(u,u_t, u_x, u_{tx},\dots)=0$
is said to be multi-symplectic if 
it can be written as a system of first-order equations
\begin{align}\label{ms}
Mz_t + Kz_x = 	\nabla _z S(z),
\end{align}
with $z\in\bbR^d$ a vector of state variables,
in which the original variable $u$ is included as one of its components.
The constant matrices $M,K\in\bbR^{d\times d}$ are skew-symmetric,
and $S$ is a smooth function depending on $z$.

A key property of the multi-symplecticity is that 
there is a multi-symplectic conservation law
\begin{align}\label{mscl}
\pt \omega + \px \kappa = 0,
\end{align}
where $\omega$ and $\kappa$ are differential two-forms defined by
\begin{align*}
\omega = \frac12\rmd z \wedge M\rmd z, \qquad \kappa= \frac12\rmd z \wedge K \rmd z.
\end{align*}
A multi-symplectic PDE also has local conservation laws
\begin{align}
\pt E(z) + \px F(z) &= 0, \label{lcl1} \\
\pt I(z) + \px G(z) &= 0, \label{lcl2}
\end{align}
where 
\begin{alignat}{2}
E (z)&= S(z) + \frac{1}{2}z_x^\top K z, \qquad F(z) &=  -\frac{1}{2}z_t^\top K z,
\\
G (z) &=  S(z) + \frac{1}{2}z_t^\top M z, \qquad I(z) &=  -\frac{1}{2}z_x^\top M z.
\end{alignat}
Integrating these local conservation laws over the spatial domain,
under appropriate boundary conditions and appropriate assumptions of $F(z)$ and $G(z)$,
leads to the global conservation laws
\begin{align} \label{msgcl}
\frac{\rmd}{\rmd t}\calE = \frac{\rmd}{\rmd t}\int E(z)\,\rmd x=0, \qquad
\frac{\rmd}{\rmd t}\calI = \frac{\rmd}{\rmd t}\int I(z)\,\rmd x=0.
\end{align}
These quantities are called energy and momentum, respectively.

\subsection{Multi-symplectic discretizations}
A numerical scheme is called multi-symplectic,
if it satisfies a discrete version of the multi-symplectic conservation law \eqref{mscl}.
As typical multi-symplectic schemes,
we give two well-known examples: the Euler box scheme and the Preissmann box scheme.

Let us introduce a splitting of two matrices $M$ and $K$,
i.e. $M=M_++M_-$ and $K=K_++K_-$ so that $M_+^\top = -M_-$
and $K_+^\top = -K_-$.
The so called Euler box scheme reads
\begin{align}\label{Eb1}
M_+ \dt^+ z_n^i + M_- \dt^- z_n^i + K_+ \dx^+ z_n^i + K_- \dx^- z_n^i = 
\nabla _z S(z_n^i). 
\end{align}
Although the above splitting is not unique,
if we choose $M_+= \frac{1}{2}M$ and $K_+=\frac{1}{2}K$, the scheme \eqref{Eb1}
is simplified to
\begin{align}\label{Eb2}
M\dt z_n^i + K \dx z_n^i = \nabla _z S(z_n^i).
\end{align}
Hereafter, we consider this special case just for simplicity.
The Euler box scheme \eqref{Eb2} satisfies the
discrete multi-symplectic conservation law
\begin{align}
\dt^+ \omega _n^i  + \dx^+ \kappa_n^i = 0, 
\end{align}
where
\begin{align}
\omega_n^i = \frac12\rmd z_n^{i-1} \wedge M \rmd z_n^i,
\qquad
\kappa_n^i = \frac12\rmd z_{n-1}^i \wedge K \rmd z_n^i.
\end{align}

The Preissmann box scheme reads
\begin{align}\label{Pb}
M\dt^+ z_{n+\frac12}^i + K \dx^+ z_n^{i+\frac12} = \nabla _z S(z_{n+\frac12}^{i+\frac12}).
\end{align}
The Preissmann box scheme \eqref{Pb} satisfies the
discrete multi-symplectic conservation law
\begin{align}
\dt^+ \omega _{n+\frac12}^i  + \dx^+ \kappa_n^{i+\frac12} = 0, 
\end{align}
where
\begin{align}
\omega_{n+\frac12}^i = \frac12\rmd z_{n+\frac12}^i \wedge M \rmd z_{n+\frac12}^i,
\qquad
\kappa_n^{i+\frac12} = \frac12\rmd z_n^{i+\frac12} \wedge K \rmd z_n^{i+\frac12}.
\end{align}
The central idea of the Preissmann box scheme is 
to apply the midpoint rule to the both time and space variables,
but the Preissmann box scheme differs from the standard midpoint rule in that
the domain should be divided with odd grid points, i.e. $N$ is odd, for the Preissmann box scheme
to ensure the uniqueness of numerical solutions.

In general, multi-symplectic integrators do not inherit the local conservation laws \eqref{lcl1}, \eqref{lcl2}
and the global conservation laws \eqref{msgcl},
except for the special cases  $S(z)$ is quadratic~\cite{br01}.
However, it should be noted that backward error analysis shows that
these global invariants are nearly preserved without any drift~\cite{mr03}:
for the Euler box and Preissmann box schemes, errors in energy and momentum are bounded
with the order $O(\Dt^2+\Dx^2)$ independently of time.
We also note that a semi-discrete scheme may inherit one of the global invariants:
for example, if we apply the Preissmann box scheme to only the spatial variable, 
the corresponding semi-discrete scheme $M\pt z_{n+1/2} + K \dx^+z_n = \nabla_z S(z_{n+1/2})$
inherits the local and global energy conservation laws~\cite{mo03}.

\section{Proposed numerical schemes}
\label{sec3}
\subsection{A reformulation of the Benjamin-type equations}
\label{sec31}

Since the multi-symplecticity tells us both local and global properties of the PDEs,
ideally we hope to find a multi-symplectic formulation for the Benjamin-type equations.
This seems, however, rather demanding,
at least to the present authors,
due to the non-local operators ($H$ and $L$).
Note that the presence of non-local operators
does not always deny the existence of a multi-symplectic structure.
Indeed, some non-local PDEs such as the Camassa--Holm and Hunter--Saxton equations possess
multi-symplectic structures~\cite{cmr14,cor08}.
The key to the success there is that the non-local operators appear only as some inverse
of standard differential operators, and can be formally eliminated by multiplying the associated
operators;
for example, the Camassa--Holm equation involves $(1-\px^2)^{-1}$.
On the other hand, since the non-locality of the Benjamin equation is due to the Hilbert transform,
whose inverse is also non-local and thus difficult to treat,
it is unlikely that there exists a multi-symplectic formulation for the Benjamin equation.
 
The above difficulty motivates us to extend the multi-symplectic formulation \eqref{ms}
so that the extended formulation can fit into the Benjamin equation.
In the formulation \eqref{ms}, $\nabla_z$ denotes a standard gradient in the finite dimensional setting.
We change this gradient with the functional derivative in the infinite dimensional setting,
and then consider the formulation
\begin{align}\label{ems}
Mz_t + Kz_x = \fracdel{\calS}{z},
\end{align}
where the right hand side denotes the functional derivative of the functional 
\begin{align}
\calS (z) = \int S(z)\,\rmd x
\end{align}
with respect to $z$. 
This idea is motivated by \cite{br97},
where a similar approach has been already mentioned.

\begin{remark}
More rigorously, in the finite dimensional setting,
for $S:\bbR^d\to\bbR$, the gradient $\nabla _z $ in $\bbR^d$ is 
defined by
\begin{align*}
S^\prime (z,\eta) = (\nabla_z S(z),\eta) \qquad
\text{for all } \eta \in \bbR^d,
\end{align*}
where $S^\prime$ denotes the G\^ateaux derivative
and $(\cdot,\cdot)$ is the inner product in $\bbR^d$.
In the infinite dimensional setting,
for $\calS: (L^2 (\bbT))^d\to\bbR$,
the gradient, i.e. functional derivative, $\delta /\delta z$ in $(L^2 (\bbT))^d$ is defined by 
\begin{align*}
\calS^\prime  (z,\eta) = \paren*{ \fracdel{\calS}{z} (z),\eta} \qquad
\text{for all } \eta \in (L^2 (\bbT))^d,
\end{align*}
where $\calS^\prime$ denotes the G\^ateaux derivative
and $(\cdot,\cdot)$ is the inner product in $(L^2 (\bbT))^d$.
For example, if
\begin{align}
\calS (u) = \int_\bbT \frac{\alpha}{2} uLu\,\rmd x,
\end{align}
the simple calculation
\begin{align}
\calS (u+\delta u) - \calS(u) = \int_\bbT \frac{\alpha}{2}
(uL\delta u + \delta u Lu)\,\rmd x + O ((\delta u)^2)
= \int_\bbT(\alpha Lu) \delta u\,\rmd  x  + O((\delta u)^2)
\end{align}
shows that
\begin{align*}
\fracdel{\calS}{u} = \alpha L u.
\end{align*}
\end{remark}

The Benjamin equation can be written as the formulation \eqref{ems}
with $z=[u,\phi,w,v]^\top$,
\begin{align}\def\arraystretch{1.3}
M = \begin{bmatrix}
0 & \frac12 & 0 & 0 \\
-\frac12 & 0 & 0 & 0 \\
0 & 0 & 0 & 0\\
0 & 0 & 0 & 0
\end{bmatrix}, \qquad
K = \begin{bmatrix}
0 & 0 & 0 & -\beta \\
0 & 0 & 1 & 0 \\
0 & -1 & 0 & 0\\
\beta & 0 & 0 & 0
\end{bmatrix}
\end{align}
and
\begin{align*}
S(z) = -wu - \frac{\gamma}{2}u^2  - \frac{\lambda}{6}u^3 + \frac{\alpha}{2}uLu + \frac{\beta}{2}v^2.
\end{align*}
Because of the symmetry of the operator $L$, the functional derivative is calculated to be
\begin{align*}
\fracdel{\calS}{z} = \begin{bmatrix}
-w - \gamma u - \frac{\lambda}{2}u^2 + \alpha Lu, \  0 ,\  -u ,\ \beta v
\end{bmatrix}^\top,
\end{align*}
and thus the formulation can be written in the
￼￼￼￼￼componentwise fashion
\begin{align}
\frac12 \phi_t -\beta v_x &= -w - \gamma u -\frac{\lambda}{2}u^2 + \alpha Lu,  \\
-\frac12 u_t + w_x &= 0, \\
-\phi_x &= -u , \\
\beta u_x &= \beta v.
\end{align}

Below, we discuss local and global properties of the Benjamin equation based on the formulation \eqref{ems}.

Let us first consider the multi-symplectic conservation law \eqref{mscl}.
The variational equation associated with \eqref{ems} is 
\begin{align}
M \rmd z_t + K \rmd z_x = \rmd \paren*{\fracdel{\calS}{z}}.
\end{align}
Here, the right hand side is calculated to be
\begin{align}\def\arraystretch{1.3}
\rmd \paren*{\fracdel{\calS}{z}} =
\begin{bmatrix}
- \gamma -\lambda u +\alpha L  & 0 & -1 & 0 \\
0 & 0  & 0 & 0 \\
-1 & 0 & 0 & 0 \\
0 & 0 & 0 & \beta
\end{bmatrix} \rmd z.
\end{align}
Since
\begin{align*}
\omega_t 
&= \frac12\rmd z_t \wedge M\rmd z + \frac{1}{2} \rmd z \wedge M \rmd z_t \\
&= \frac12 \paren*{K \rmd z_x - \rmd \paren*{\fracdel{\calS}{z}}} \wedge \rmd z
 -  \frac12 \rmd z \wedge \paren*{K \rmd z_x - \rmd \paren*{\fracdel{\calS}{z}}}  \\
&= -\frac12 \paren*{\rmd z \wedge K \rmd z}_x -\frac12 \alpha L \rmd u\wedge\rmd u + \frac12\rmd u \wedge \alpha L \rmd u \\
&= -\kappa_x + \alpha \rmd u \wedge L \rmd u,
\end{align*}
we have 
\begin{align}\label{bmscl}
\omega _t + \kappa_x = \alpha \rmd u \wedge L \rmd u.
\end{align}
Here $\rmd u \wedge L \rmd u$ does not vanish, and thus the multi-symplectic conservation law \eqref{mscl}
does not hold for the Benjamin equation.
However, integrating \eqref{bmscl} over the spatial domain under the periodic boundary condition,
we obtain the global property
\begin{align} \label{bmsgcl}
\frac{\rmd}{\rmd t} \int_\bbT \omega\,\rmd x = 0.
\end{align}

Next we consider the local conservation laws.
Taking the inner product of \eqref{ems} with $z_t$, we obtain
\begin{align}\label{eqks}
z_t^\top Kz_x = z_t^\top \fracdel{\calS}{z}
\end{align}
because of the skew-symmetry of $M$.
Noticing that
\begin{align}
z_t^\top Kz_x = \px \paren*{\frac12 z_t^\top Kz} - \pt \paren*{\frac12 z_x^\top Kz}
\end{align}
and
\begin{align}
z_t^\top \fracdel{\calS}{z} = \pt S(z) + \frac{\alpha}{2} u_t Lu - \frac{\alpha}{2}uLu_t,
\end{align}
we have
\begin{align} \label{emslcl1}
\pt E(z) + \px F(z) = -\frac{\alpha}{2} u_t Lu + \frac{\alpha}{2} uLu_t.
\end{align}
Therefore, the local conservation law \eqref{lcl1} does not hold.
Similarly, we have
\begin{align} \label{emslcl2}
\pt I(z) + \px G(z) = -\frac{\alpha}{2} u_x Lu + \frac{\alpha}{2} uLu_x,
\end{align}
and thus the local conservation law \eqref{lcl2} is not satisfied.
Although the local conservation laws do not hold,
integrating \eqref{emslcl1} and \eqref{emslcl2} over the spacial domain
under the periodic boundary condition immediately leads to the global conservation laws
\begin{align}
\frac{\rmd}{\rmd t} \calE &= \frac{\rmd}{\rmd t} \int_{\bbT} E(z)\,\rmd x
=\frac{\rmd}{\rmd t} \int_{\bbT} \paren*{- \frac{\gamma}{2}u^2 - \frac{\lambda}{6}u^3 + \frac{\alpha}{2}uLu
-\frac{\beta}{2}u_x^2}\,\rmd x = 0, \label{bgcle} \\
\frac{\rmd}{\rmd t} \calI &= \frac{\rmd}{\rmd t} \int_{\bbT} I(z)\,\rmd x
= \frac{\rmd}{\rmd t} \int_{\bbT} - \frac12 u^2 \,\rmd x = 0. \label{bgcli}
\end{align}
Here, the symmetry of  the operator $L$
\begin{align}
\int_\bbT u L v\,\rmd x = \int_\bbT (Lu)v\,\rmd x
\end{align}
is used.

\begin{remark}
Choosing $\beta=\gamma=0$ leads to the formulation \eqref{ems}
for the Benjamin--Ono equation:
\begin{align*}\def\arraystretch{1.3}
\begin{bmatrix}
0 & \frac12 & 0 \\
-\frac12 & 0 & 0 \\
0 & 0 & 0
\end{bmatrix}
\begin{bmatrix}
u \\ \phi \\ w
\end{bmatrix}_t
+ \begin{bmatrix}
0 & 0 & 0 \\
0 & 0 & 1 \\
0 & -1 & 0
\end{bmatrix}
\begin{bmatrix}
u \\ \phi \\ w
\end{bmatrix}_x
= \fracdel{}{z}
\int \paren*{-wu - \frac{\lambda}{6}u^3 + \frac{\alpha}{2}uLu}\,\rmd x.
\end{align*} 
Similarly, choosing $\alpha=\gamma=0$ leads to the formulation \eqref{ems}
for the KdV equation:
\begin{align}\def\arraystretch{1.3}
\begin{bmatrix}
0 & \frac12 & 0 & 0 \\
-\frac12 & 0 & 0 & 0 \\
0 & 0 & 0 & 0\\
0 & 0 & 0 & 0
\end{bmatrix}
\begin{bmatrix}
u \\ \phi \\ w \\ v
\end{bmatrix}_t +
\begin{bmatrix}
0 & 0 & 0 & -\beta \\
0 & 0 & 1 & 0 \\
0 & -1 & 0 & 0\\
\beta & 0 & 0 & 0
\end{bmatrix}
\begin{bmatrix}
u \\ \phi \\ w \\ v
\end{bmatrix}_x =
\fracdel{}{z}
\int \paren*{-wu - \frac{\lambda}{6}u^3 + \frac{\beta}{2}v^2}\,\rmd x,
\end{align}
but this coincides with the well-known multi-symplectic form~\cite{am04,am05,zq00}
since in this case $\delta \calS /\delta z = \nabla_z S(z)$. 
\end{remark}

%So far, no one has ever found a multi-symplectic formulation for the Benjamin equation.

\subsection{Discretizations of the operators $H$ and $L$}
\label{sec32}
To solve the Benjamin equation numerically,
it is mandatory to discretize the operators $H$ and $L=H\px$.
In particular, the operator $L$ should be discretized so that the symmetry is kept in the discrete setting.

We first review the approach developed by Thom{\'e}e--Vasudeva Murthy~\cite{tm98}.
However, as will be explained soon, 
their approach only makes sense when the domain is divided into even intervals. 
Therefore, we shall develop a new discretization method
for odd intervals so that the Preissmann box scheme is applicable.

\subsubsection{Thom{\'e}e--Vasudeva Murthy's approach for $H$ (even intervals)}
In~\cite{tm98} a discrete version of the Hilbert transform \eqref{htf2} for $\bmu=\{u_n\}_{n\in\bbZ}\in\disT$ 
is defined by
\begin{align}
(\dht \bmu)_n = \frac{1}{l}
\sum_{j=0}^{N/2-1} \cot \paren*{\frac{\pi	}{l}(2j+1)\Dx}
u_{n-(2j+1)} 2 \Dx.
\end{align}
Here, the midpoint rule is used for each interval $[x_{2j},x_{2j+2}]$ ($j=0,\dots,N-1$).
This definition is rewritten as a discrete convolution
\begin{align}\label{conv1}
(\dht \bmu)_n = \sum_{j=0}^{N-1} c_{n-j}u_j,
\qquad
\text{where} 
\qquad c_n=
\begin{cases}
\frac{2}{N} \cot \paren*{\frac{\pi}{N}n}, & \text{if }n\text{ is odd},
\\ 
0, & \text{if }n\text{ is even}.
\end{cases}
\end{align}
Here we note that $\{c_n\}_{n\in\bbZ}\in\disT$ and $c_n=-c_{-n}$.
Furthermore, we rewrite \eqref{conv1} by using the discrete Fourier transform.
For $\bmu=\{ u_n\}_{n\in\bbZ}\in\disT$, the discrete Fourier transform $\dft \bmu$ is defined by
\begin{align}
(\dft \bmu) _k = \sum_{n=0}^{N-1}u_n \rme^{-2\pi \rmi n k/N},
\end{align}
which also belongs to $\disT$.
For $\bmv =\{ v_k\}_{k\in\bbZ}\in\disT$, the inverse transform $\dft^{-1} \bmv$ is defined by
\begin{align}
(\dft ^{-1}\bmv)_n = \frac{1}{N}\sum_{k=0}^{N-1}v_k \rme^{2\pi \rmi n k/N}.  
\end{align}

\begin{lemma}[\cite{tm98}] \label{tm:lem41}
The discrete Hilbert transform defined in \eqref{conv1}
is expressed as
\begin{align}
\dht \bmu = \dft^{-1} (-\rmi \Se) \dft \bmu,
\end{align}
where $\Se$ is defined by
\begin{align}
%\Se = 
%\begin{matrix}
%\vphantom{0}\\
%\coolleftbrace{\frac{N}{2}-1}{1 \\ \ddots \\ 1}\\
%\vphantom{0}\\
%\coolleftbrace{\frac{N}{2}-1}{-1 \\ \ddots \\ -1}
%\end{matrix}%
%\begin{bmatrix}
%0 &  &  &  &  &  & &\\
% & 1 &  &  &  &  &  &\\
% &  & \ddots &  &  &  & & \\
% &  &  & 1 &  &  &  & \\
% &  &  &  & 0 &  &  & \\
% &  &  &  &  & -1 &  & \\
% &  &  &  &  &  & \ddots & \\
% & & & & & & & -1
%\end{bmatrix}
%\in
%\bbR^{N\times N}. \\
\Se = \diag (0,\underbrace{1,\dots,1}_{\frac{N}{2}-1},0,\underbrace{-1,\dots,-1}_{\frac{N}{2}-1}) .
\end{align}
%\begin{align*}
%(\dht u)_n = DF^{-1}
%[-\rmi \widetilde{\sgn}_{\rme}(k)DF[u_n]_k]_n,
%\end{align*}
%where
%\begin{align}
%\widetilde{\sgn}_{\rme}(k) =
%\begin{cases}
%1 & (1 \leq k \leq N/2 -1),
%\\
%-1 & (N/2 + 1 \leq k \leq N-1),
%\\
%0 & (k = 0 , \ N/2).
%\end{cases}
%\end{align}
\end{lemma}

The following lemma indicates that
$\dht \bmu$
is a second order approximation to $Hu$.

\begin{lemma}[\cite{tm98}] \label{tm:lem42}
Assume that $u$ is periodic and sufficiently smooth.
Then it follows that
\begin{align}
\| \dht \bmu - H u\| _\infty \leq C (\Dx)^2 \| u \| _{C^3},
\end{align}
where 
$Hu$ is an abbreviation of $[Hu(x_0),Hu(x_1),\dots,Hu(x_{N-1})]^\top$,
$\|\bmu\|_\infty = \max_n \|u_n\|$ and
$\|u\|_{C^3} = \max_x |u(x)| + \max_x |u^\prime(x)| + 
\max_x |u^{\prime\prime}(x)| + \max_x |u^{\prime\prime\prime}(x)|$.
\end{lemma}

\subsubsection{New approach for $H$ (odd intervals)}

In above, we defined a discrete version of the operator $H$ for even $N$,
and discussed its accuracy.
For the Preissmann box scheme, however, we need its odd number
counterpart.
Recall that 
the discrete Hilbert transform is written through the discrete Fourier transform in \autoref{tm:lem41}.
We wish to obtain a similar expression for odd $N$.
For this aim, it is convenient to split \eqref{htf2} into two terms:
\begin{align}
Hu(x)
&= \frac{1}{l} \pv \int_0^l \cot \paren*{\frac{\pi}{l}y}u(x-y)\,\rmd y \nonumber \\
&= \frac{1}{2l} \pv \int_0^l \paren*{\cot \paren*{\frac{\pi}{2l}y} - \tan \paren*{\frac{\pi}{2l}y}} 
u(x-y) \,\rmd y \nonumber \\
&= \frac{1}{2l} \lim_{\epsilon\to +0} 
\paren*{ \int_\epsilon^l  \cot \paren*{\frac{\pi}{2l}y} u(x-y)\,\rmd y
- \int_0^{l-\epsilon} \tan\paren*{\frac{\pi}{2l}y} u(x-y) \,\rmd y} .  \label{HT1}
\end{align}
For the first term, we apply the midpoint rule to each interval $[(2j-2)\Dx, 2j\Dx]$ ($j=1,\dots, (N-1)/2$)
and the rule $\int_a^b f(x)\,\rmd x \approx (b-a) f(b)$ to the remaining interval $[(N-1)\Dx,l]$.
For the second term, we apply the midpoint rule to each interval $[(2j-1)\Dx, (2j+1)\Dx]$ ($j=1,\dots, (N-1)/2$)
and the rule $\int_a^b f(x)\,\rmd x \approx (b-a) f(a)$ to the remaining interval $[0,\Dx]$.
We now define a discrete Hilbert transform by
\begin{align}
(\dht \bmu)_n 
&= \frac{1}{2l} \sum_{j=1}^{(N-1)/2} \cot \paren*{\frac{\pi}{2l} (2j-1)\Dx} u_{n-(2j-1)}2\Dx \\
& \phantom{=} + \frac{1}{2l} \cot \paren*{\frac{\pi}{2}}u_{n-N}\Dx \\
& \phantom{=} - \frac{1}{2l} \tan (0) u_n\Dx \\
& \phantom{=} - \frac{1}{2l} 
	\sum_{j=1}^{(N-1)/2} \tan \paren*{\frac{\pi}{2l} 2j\Dx} u_{n-2j}2\Dx \\
&= \frac{1}{2l} \sum_{j=1}^{(N-1)/2} \cot \paren*{\frac{\pi}{2l} (2j-1)\Dx} u_{n-2j+1}2\Dx \\
& \phantom{=} - \frac{1}{2l} 
	\sum_{j=1}^{(N-1)/2} \tan \paren*{\frac{\pi}{2l} 2j\Dx} u_{n-2j}2\Dx. \label{DHT}
\end{align}
This is rewritten as a discrete convolution
\begin{align} \label{conv2}
(\dht \bmu)_n = \sum_{j=0}^{N-1} d_{n-j}u_j, \qquad
\text{where} 
\qquad d_n=
\begin{cases}
\frac{1}{N} \cot \paren*{\frac{\pi n}{2N}}, & \text{if }n \text{ is odd},
\\ 
-\frac{1}{N} \tan \paren*{\frac{\pi n}{2N}}, & \text{if }n \text{ is even},
\end{cases}
\end{align}
and thus it can be expressed through the discrete Fourier transform as desired.

\begin{lemma}
The discrete Hilbert transform defined in \eqref{conv2}
is expressed as
\begin{align}
\dht \bmu = \dft^{-1} (-\rmi \So) \dft \bmu,
\end{align}
where $\So$ is defined by
\begin{align}
\So = \diag (0,\underbrace{1,\dots,1}_{\frac{N-1}{2}},\underbrace{-1,\dots,-1}_{\frac{N-1}{2}}) .
\end{align}
%\begin{align*}
%(\dht u)_n = DF^{-1}
%[-\rmi \widetilde{\sgn}_{\rmo}(k)DF[u_n]_k]_n,
%\end{align*}
%where
%\begin{align}
%\widetilde{\sgn}_{\rmo}(k) =
%\begin{cases}
%1 & (1 \leq k \leq (N-1)/2 ),
%\\
%-1 & ((N+1)/2  \leq k \leq N-1),
%\\
%0 & (k = 0 ).
%\end{cases}
%\end{align}
\end{lemma}

\begin{proof}
The proof is similar to that of \autoref{tm:lem41}.
Since $(\dht \bmu)_n$ is written as a convolution in \eqref{conv2},
it immediately follows that$(\dft \dht \bmu)_n = (\dft \bmd)_n (\dft \bmu)_n$.
Hence, we wish to prove that $(\dft \bmd)_n = -\rmi \widetilde{\sgn} (n)$,
or equivalently $\dft^{-1}(-\rmi \widetilde{{\mathrm{\bf sgn}}})=\bmd$, where
$\widetilde{{\mathrm{\bf sgn}}}\in\disT$ and its components are
\begin{align*}
\widetilde{\sgn}(n) = \begin{cases}
0  & \text{if }n = 0, \\
1 & \text{if }1\leq n \leq \frac{N-1}{2}, \\
-1 & \text{if }\frac{N+1}{2}\leq n\leq N-1.
\end{cases}
\end{align*}
It follows that
\begin{align*}
\dft^{-1}(-\rmi \widetilde{{\mathrm{\bf sgn}}})_n
&=
\frac{-\rmi}{N} \sum_{k=0}^{N-1} \widetilde{\sgn} (n) \rme^{2\pi \rmi nk/N}
=
\frac{-\rmi}{N}\paren*{
\sum_{k=1}^{(N-1)/2} \rme^{2\pi \rmi nk/N} - \sum_{k=(N+1)/2}^{N-1} \rme^{2\pi \rmi nk/N}
}\\
&=
\frac{-\rmi}{N} \sum_{k=1}^{(N-1)/2} \paren*{\rme^{2\pi \rmi nk/N} - \rme^{-2\pi \rmi nk/N}}
=
\frac{2}{N} \sum_{k=1}^{(N-1)/2}  \sin\paren*{\frac{2\pi  nk}{N}} \\
&=
\frac{1}{N}\frac{\cos \paren*{\frac{\pi n}{N}} - (-1)^n}{2\sin \paren*{\frac{\pi n}{2N}} \cos \paren*{\frac{\pi n}{2N}}} \\
&=
\begin{cases}
\frac{1}{N} \cot \paren*{\frac{\pi n}{2N}},  & \text{if } n \text{ is odd},\\
-\frac{1}{N} \tan \paren*{\frac{\pi n}{2N}}, &  \text{if } n \text{ is even}
\end{cases} \\
&=
d_n
\end{align*}
\end{proof}

\begin{lemma}
Assume that $u$ is periodic and sufficiently smooth.
Then it follows that
\begin{align}
\| \dht \bmu - H u\| _\infty \leq C (\Dx)^2 \| u \| _{C^3},
\end{align}
where 
$Hu$ is an abbreviation of $[Hu(x_0),Hu(x_1),\dots,Hu(x_{N-1})]^\top$,
$\|\bmu\|_\infty = \max_n \|u_n\|$ and
$\|u\|_{C^3} = \max_x |u(x)| + \max_x |u^\prime(x)| + 
\max_x |u^{\prime\prime}(x)| + \max_x |u^{\prime\prime\prime}(x)|$.
\end{lemma}

\begin{proof}
The proof is similar to that of \autoref{tm:lem42}.
First, we re-express the continuous and discrete Hilbert transforms (\eqref{HT1} and \eqref{DHT})
as follows:
\begin{align*}
Hu(x) &=
\frac{1}{2l} \lim_{\epsilon\to+0} \int_\epsilon^l \cot \paren*{\frac{\pi y}{2l}}
\paren*{u(x-y)-u(x+y)}\,\rmd y, \\
(\dht \bmu)_n &=
\frac{1}{2l}
\sum_{m=1}^{(N-1)/2} 2\Dx
\cot \paren*{\frac{\pi x_{2m-1}}{2l}}
\paren*{u_{n-(2m-1)} - u_{n+(2m-1)}}.
\end{align*}
We here define three function $\psi (y)$, $\phi(x,y)$ and $\delta (x,y)$ by
\begin{align}
\psi (y) &= y\cot(\pi y), \\
\phi(x,y) &= \frac{u(x-y)-u(x+y)}{y}, \\
\delta (x,y) &= \psi \paren*{\frac{y}{2l}} \phi(x,y) = \frac{1}{2l} \cot \paren*{\frac{\pi y}{2l}} \paren*{u(x-y)-u(x+y)}.
\end{align}
Below, by using these functions, we show that $| (\dht \bmu)_n - Hu(x_n) | \leq C (\Dx)^2 \| u \| _{C^3}$.
Hereafter, we promise that $C$ is a generic constant independently of $\Dx$ and $u$.
First, we rewrite $| (\dht \bmu)_n - Hu(x_n) |$ as follows.
\begin{align}
&| (\dht \bmu)_n - Hu(x_n) | \\
&= 
\left|
\frac{1}{2l}
\sum_{m=1}^{(N-1)/2} 2\Dx
\cot \paren*{\frac{\pi x_{2m-1}}{2l}}
\paren*{u_{n-(2m-1)} - u_{n+(2m-1)}} -
\frac{1}{2l} \lim_{\epsilon\to+0} \int_\epsilon^l \cot \paren*{\frac{\pi y}{2l}}
\paren*{u(x_n-y)-u(x_n+y)}\,\rmd y
\right| \\
&=
\left|
\sum_{m=1}^{(N-1)/2} \int_{x_{2m-2}}^{x_{2m}} \paren*{\delta (x_n,x_{2m-1}) - \delta (x_n,y)}\,\rmd y
+ \int_{x_{N-1}}^{x_N} \paren*{\delta (x_n,x_{N}) - \delta (x_n,y)}\,\rmd y
\right|.
\end{align}
By the standard error estimate for the midpoint and rectangular rules, there exist constants
$x_{2m-2}<\xi_m<x_{2m}$ ($m=1,\dots,(N-1)/2$) and $x_{N-1}<\eta<x_N$, and we proceed with the estimate
\begin{align}
&| (\dht \bmu)_n - Hu(x_n) | \\
&=
\left|
\sum_{m=1}^{(N-1)/2} \frac{(2\Dx)^3}{24} \fracpar{^2\delta}{y^2} (x_n,\xi_m)
+ \frac{(\Dx)^2}{2} \fracpar{\delta}{y} (x_n,\eta) 
\right| \\
&\leq
\max_{0\leq y \leq l} \left| \fracpar{^2\delta}{y^2} (x_n,y) \right|
\sum_{m=1}^{(N-1)/2} \frac{(\Dx)^3}{3}
+\max_{0\leq y \leq l} \left| \fracpar{\delta}{y} (x_n,y) \right| \frac{(\Dx)^2}{2} \\
&\leq
\max_{0\leq y \leq l} \left| \fracpar{^2\delta}{y^2} (x_n,y) \right|
\frac{l(\Dx)^2}{6}
+\max_{0\leq y \leq l} \left| \fracpar{\delta}{y} (x_n,y) \right| \frac{(\Dx)^2}{2}.
\end{align}
Here we note that $\psi$ is independent of $u$ and thus its norms are bounded.
Thus we have
\begin{align}
\max_{0\leq y \leq l} \left| \fracpar{^2\delta}{y^2} (x_n,y) \right| \leq C \| \phi (x_n,\cdot) \| _{C^2},
\qquad 
\max_{0\leq y \leq l} \left| \fracpar{\delta}{y} (x_n,y) \right| \leq C \| \phi (x_n,\cdot) \| _{C^1}.
\end{align}
By the Taylor expansion,
\begin{align*}
\| \phi (x_n,\cdot) \| _{C^2} \leq 2\|u\| _{C^3}, \qquad \| \phi (x_n,\cdot) \| _{C^1} \leq 2\|u\| _{C^3}
\end{align*}
from which we obtain $| (\dht \bmu)_n - Hu(x_n) | \leq C (\Dx)^2 \| u \| _{C^3}$.
\end{proof}

\subsubsection{Symmetric discretization of $L$}

We now define a discrete version of the operator $L=H\px$.
Since the discrete Hilbert transform defined above is formulated in terms of the discrete Fourier transform,
one convenient way for discretizing $\px$ is to employ
the spectral difference, i.e.,
for $\bmu = \{u_n\}\in\disT$, 
\begin{align} \label{spec}
\ddo^\infty \bmu = \dft^{-1} \paren*{ \rmi \frac{2\pi}{l} \tilde{K} }\dft \bmu ,
\end{align}
where $\tilde{K}\in\bbR^{N\times N}$ is defined by
\begin{align}
\tilde{K} = \begin{cases}
\diag\paren*{0,1,\dots,\frac{N}{2}-1,0,-\paren*{\frac{N}{2}-1},\dots,-1}, & \text{if } N \text{ is even}, \\
\diag\paren*{0,1,\dots,\frac{N-1}{2},-\paren*{\frac{N-1}{2}},\dots,-1}, & \text{if } N \text{ is odd}.
\end{cases}
\end{align}
Note that this is a real matrix (see, for example,~\cite{fo96}).

%We note that the definition \eqref{spec} can be expressed as a discrete convolution:
%\begin{align}
%(\ddo^\infty \bmu)_n = \sum_{j=0}^{N-1} p_{n-j}u_j, 
%\end{align}
%where if $N$ is even,
%\begin{align}
%p_n = \begin{cases}
%\frac{\pi}{l} (-1)^n \cot \paren*{\frac{\pi}{N}n},  & \text{if } n\neq 0,    \\
%0, & \text{if }n=0,
%\end{cases}
%\end{align}
%and if $N$ is odd,
%\begin{align}
%p_n = \begin{cases}
%\frac{\pi}{l} \frac{(-1)^n}{\sin \paren*{\frac{\pi}{N}n}},  & \text{if } n\neq 0,   \\
%0, & \text{if }n=0.
%\end{cases}
%\end{align}

We now define an operator $\dlt$ by
\begin{align}
\dlt \bmu = \dft^{-1}
\paren*{
\frac{2\pi}{l} \tilde{K}
S} \dft \bmu,
\end{align}
where $S=\Se$ if $N$ is even, and $S=\So$ if $N$ is odd.
This is a symmetric matrix as the next lemma shows, which is crucial
in the subsequent theoretical analyses.

\begin{lemma}
The operator $\dlt$ is symmetric in the sense that ${\dlt}^\top = \dlt$.
\end{lemma}

\begin{proof}
Note that $H_{\Delta x}$ is a real matrix for both even and odd $N$,
which is clear from the original definition.
As noted above, the spectral difference operator is also real.
Thus the matrix $\dlt$ is real, and it suffices to show 
$(\dlt)^{*} = \dlt$, where $(\,\cdot\,)^*$ denotes the Hermitian conjugate.
But it is clear since $\tilde{K}$ and $S$ are both real diagonal matrices.
\end{proof}

%\begin{proof}
%	Due to $c_{-n}=-c_n$ and $p_{-n}=-p_n$, we have
%	\begin{align*}
%	\sum_{j=0}^{N-1}u_j (\dlt \bmv)_j
%	=
%	\sum_{n=0}^{N-1}\sum_{m=0}^{N-1}\sum_{j=0}^{N-1}u_{j} p_{j-m}c_{m-n}v_n
%	=
%	\sum_{n=0}^{N-1}\sum_{m=0}^{N-1}\sum_{j=0}^{N-1}v_n c_{n-m}p_{m-j}u_{j}
%	=
%	\sum_{n=0}^{N-1}v_n (\dlt  \bmu)_n.
%	\end{align*}
%\end{proof}

\subsection{An Euler box scheme}
\label{sec33}
We apply the Euler box scheme \eqref{Eb2} to the formulation \eqref{ems}
(just for simplicity, we consider the simplified version only).
To do this, we introduce the notation
\begin{align} \label{bds}
\fracdel{\calS}{z}(z_n^i)
= 
\begin{bmatrix}
-w_n^i 
-\gamma u_n^i -\frac{\lambda}{2}(u_n^i)^2 + \frac{\alpha}{2}u_n^i (\dlt \bmu^i)_n 
\\
0
\\
-u_n^i
\\
\beta v_n^i
\end{bmatrix}.
\end{align}
The Euler box scheme for the Benjamin equation reads
\begin{align} \label{bEb2}
M \dt z_n^i + K \dx z_n^i = \fracdel{\calS}{z}  (z_n^i),
\end{align}

We consider if the global property \eqref{bmsgcl} is inherited in the discrete setting.

\begin{theorem}\label{th1:eb}
For the Euler box scheme \eqref{bEb2}, it follows that
\begin{align*}
\dt \omega_n^i + \dx \kappa _n^i = \alpha \rmd u_n^i \wedge (\dlt \rmd \bmu^i)_n
\end{align*}
and 
\begin{align}\label{bebgcl}
\dt^+ \sum_{n=0}^{N-1} \omega _n^i = 0,
\end{align}
where 
\begin{align}
\omega_n^i = \frac{1}{2} \rmd z_n^{i-1} \wedge M \rmd z_n^i, 
\qquad
\kappa_n^i = \frac{1}{2} \rmd z_{n-1}^i \wedge K \rmd z_n^i.
\end{align}
\end{theorem}

\begin{proof}
We consider the discrete variational equation
\begin{align}
M\dt \rmd z_n^i + K\dx \rmd z_n^i = \rmd \paren*{\fracdel{\calS}{z}(z_n^i)}.
\end{align}
It then follows that
\begin{align*}
\dt^+ \omega_n^i + \dx \kappa _n^i  
= \rmd z_n^i \wedge M \dt \rmd z_n^i + \rmd z_n^i \wedge K \dx \rmd z_n^i 
= \rmd z_n^i \wedge \rmd \paren*{\fracdel{\calS}{z}(z_n^i)}
= \alpha \rmd u_n^i \wedge \dlt (\rmd \bmu^i)_n,
\end{align*}
which immediately indicates \eqref{bebgcl}.
\end{proof}

Next we consider the global conservation laws \eqref{bgcle} and \eqref{bgcli}.
As is the case with standard multi-symplectic PDEs and multi-symplectic discretization methods,
such invariants cannot be preserved in the fully-discrete setting.
However, as shown below, semi-discrete schemes possess one of such invariants.

\begin{proposition}\label{prop1:eb}
For the semi-discrete scheme
\begin{align} \label{bEbs1}
M \pt z_n + K \dx z_n = \fracdel{\calS}{z} (z_n), 
\end{align}
it follows that
\begin{align} \label{bEbgcle}
\pt \sum_{n=0}^{N-1} E_n = 0, \qquad
\text{where}
\qquad
E_n = S(z_n) + \frac12 (\dx^- z_n)^\top Kz_n.
\end{align}
For the semi-discrete scheme
\begin{align} \label{bEbs2}
M\dt z^i + K \px z^i = \fracdel{\calS}{z} (z^i),
\end{align}
it follows that
\begin{align} \label{bEbgcli}
\dt^+ \int_\bbT I^i \,\rmd x = 0,
\qquad
\text{where}
\qquad
I^i = -\frac12 (\px z^{i-1})^\top M z^i.
\end{align}
\end{proposition}

\begin{proof}
We only prove \eqref{bEbgcle} because \eqref{bEbgcli} can be proved in a similar manner.

Taking the inner product \eqref{bEbs1} with $\pt z_n$, we obtain
\begin{align}
(\pt z_n)^\top K \dx^+ z_n = (\pt z_n)^\top \fracdel{\calS}{z} (z_n)
\end{align}
because of the skew-symmetry of $M$.
Noticing that
\begin{align}
(\pt z_n)^\top K \dx^+ z_n
=
\dx^+ \paren*{\frac12 (\pt z_{n-1})^\top Kz_n} - \pt \paren*{\frac12 (\dx^- z_n)^\top K z_n}
\end{align}
and
\begin{align}
(\pt z_n)^\top \fracdel{\calS}{z} (z_n)
=
\pt S(z_n) + \frac{\alpha}{2} (\pt u_n)(\dlt \bmu)_n - \frac{\alpha}{2} u_n(\dlt (\pt \bmu))_n,
\end{align}
we have
\begin{align}
\pt E_n + \dx^+ \paren*{-\frac12 (\pt z_{n-1})^\top Kz_n}
= 
- \frac{\alpha}{2} (\pt u_n)(\dlt \bmu)_n + \frac{\alpha}{2} u_n (\dlt (\pt \bmu))_n.
\end{align}
from which \eqref{bEbgcle} follows.
\end{proof}

\begin{remark}
For the fully-discrete scheme \eqref{bEb2}, the mass conservation 
\begin{align}
\dt \sum_{n=0}^{N-1} u_n^i = 0
\end{align}
is satisfied.
\end{remark}

\subsection{A Preissmann box scheme}
\label{sec34}
The Preissmann box scheme for the Benjamin equation reads
\begin{align} \label{bPb}
M \dt^+ z_{n+1/2}^i + K \dx^+ z_n^{i+1/2} = \fracdel{\calS}{z} (z_{n+1/2}^{i+1/2}).
\end{align}

\begin{theorem}
For the Preissmann box scheme \eqref{bPb}, it follows that
\begin{align*}
\dt^+ \omega_{n+1/2}^i + \dx^+ \kappa _n^{i+1/2} = \alpha \rmd u_{n+1/2}^{i+1/2} \wedge (\dlt \rmd \bmu^{i+1/2})_{n+1/2}
\end{align*}
and 
\begin{align}\label{bpbgcl}
\dt^+ \sum_{n=0}^{N-1} \omega _{n+1/2}^i = 0,
\end{align}
where 
\begin{align}
\omega_{n+1/2}^i = \frac{1}{2} \rmd z_{n+1/2}^i \wedge M \rmd z_{n+1/2}^i, 
\qquad
\kappa_n^{i+1/2} = \frac{1}{2} \rmd z_n^{i+1/2} \wedge K \rmd z_n^{i+1/2}.
\end{align}
\end{theorem}
This theorem is proved in a similar manner to \autoref{th1:eb}.

\begin{proposition}
For the semi-discrete scheme
\begin{align} \label{bPbs1}
M \pt z_{n+1/2} + K \dx^+ z_n = \fracdel{\calS}{z} (z_{n+1/2}), 
\end{align}
it follows that
\begin{align} \label{bPbgcle}
\pt \sum_{n=0}^{N-1} E_{n+1/2} = 0, \qquad
\text{where}
\qquad
E_{n+1/2} = S(z_{n+1/2}) + \frac12 (\dx^+ z_n)^\top Kz_{n+1/2}.
\end{align}
For the semi-discrete scheme
\begin{align} \label{bPbs2}
M\dt^+ z^i + K \px z^{i+1/2} = \fracdel{\calS}{z} (z^{i+1/2}),
\end{align}
it follows that
\begin{align} \label{bPbgcli}
\dt^+ \int_\bbT I^i \,\rmd x = 0,
\qquad
\text{where}
\qquad
I^i = -\frac12 (\px z^i)^\top M z^i.
\end{align}
\end{proposition}

These properties are proved in a similar manner to \autoref{prop1:eb}.

\begin{remark}
For the fully-discrete scheme \eqref{bPb}, the mass conservation 
\begin{align}
\dt^+ \sum_{n=0}^{N-1} u_n^i = 0
\end{align}
is satisfied.
\end{remark}

\section{Numerical experiments}
\label{sec4}

In this section, we check the proposed integrators via several
numerical experiments.

\subsection{Solitary wave solution for the Benjamin--Ono equation}

First, we try the Benjamin--Ono equation.
We compare the following four integrators.
\begin{itemize}
	\item The proposed Euler box scheme~\eqref{bEb2},
	\item The proposed Preissmann box scheme~\eqref{bPb},
	\item The $\mathcal{I}$-preserving scheme by Thom\'ee--Vasudeva Murthy~\cite{tm98},
	\item The Heun method with central difference.
\end{itemize}
For the readers' convenience, the Thom\'ee--Vasudeva Murthy scheme reads
\[
\delta_t^+u^i_n + f(u^{i+1/2}_n)
-{{H}}_{\Delta x}\delta_x^+\delta_x^-u^{i+1/2}_n =0,
\]
where
\begin{align*}
f(u^i_n) =
\frac{1}{6\Delta x}(u^i_{n+1}+u^i_n+u^i_{n-1})(u_{n+1}^i-u_{n-1}^i).
\end{align*}
Note that this, as well as the Preissmann box scheme,
form systems of nonlinear equations at each time step,
and we need some nonlinear solver.
We conducted all the numerical experiments on MATLAB 2010b,
and employed \texttt{fsolve} as our nonlinear solver.
The employed Heun scheme reads
\begin{align*}
\delta_t^+ u_n^i = \frac{1}{2}\left\{g(u_n^i)+g\left(u_n^i+g(u_n^i)\Delta t\right)\right\},
\quad\text{where}
\quad g(u_n^i)= -\delta_x \left\{(u_n^i)^2/2-L_{\Delta x}u_n^i \right\}.
\end{align*}
The first three integrators are in some sense structure-preserving,
and the last one is not, while keeping the same second order accuracy.
The last one is employed to see whether actually structure-preserving
integrators are advantageous.

We take the parameters to
$\alpha = 1, \beta = 0, \gamma = 0, \lambda = 1, l = 30$, and
$\Delta t = 2.5\times 10^{-3}$.
We take $\Delta x = l/256$ for the Thom\'ee--Vasudeva Murthy scheme,
which allows only even number of grid points,
and $\Delta x = l/255$ for the rest.
The initial data is chosen to $u(x_n, 0)={2cA^2}/{1-\sqrt{1-A^2}\cos(cA(x_n-l/2))}$,
where $c = 0.25$ and  $A = {2\pi}/{cl}$.
The parameter $c$ denotes the speed of the wave.
The initial data corresponds to the solitary wave solution
$u(x,t)={2cA^2}/{1-\sqrt{1-A^2}\cos(cA(x-ct-l/2))}$.

We show the results in Fig.~\ref{fig:BOsoliton}--\ref{fig:BOsolitonMomentum}.
Fig.~\ref{fig:BOsoliton} shows the wave profiles.
The three structure-preserving integrators well capture the wave
propagation, whereas the Heun scheme exhibits instability
with the same discretization widths.
This confirms the superiority of the structure-preserving integrators.
Fig.~\ref{fig:BOsolitonEnergy} and Fig.~\ref{fig:BOsolitonMomentum}
shows the evolution of the invariants $\mathcal{I}$ and $\mathcal{E}$.
From these figures we observe the following two facts.
First, the two invariants blow up in the Heun scheme,
which reflects the instability observed above.
Second, the other schemes more or less nearly preserve both invariants,
at least up to $10^{-8}$.
For the Euler box and Preissmann box schemes, this is somewhat we expected
due to the near conservation properties mentioned above.
These properties are verified by the backward error analysis (cf.~\cite{mo03,mr03}).
The key of the analysis is that the modified equation is also of the form \eqref{ems}.
%More precisely speaking, the schemes proposed in the present paper
%are not multi-symplectic schemes in the strict sense of the word---%
%they lose the local conservation laws, and hence the standard
%backward error analysis yielding rigorous near conservation of the invariants
%does not hold here (cf.~\cite{mo03}).
%Nevertheless, the numerical results suggest that the proposed schemes
%actually enjoy the near conservation property.
The Thom\'ee--Vasudeva Murthy scheme better preserves $\mathcal{I}$,
which is a natural result since it strictly conserves the invariant
by construction.
The error observed here should be attributed to the nonlinear solver
employed in the time-stepping.

The results above confirms that the proposed integrators are
in fact good discretizations.

\begin{figure}[htbp]
\centering
\includegraphics{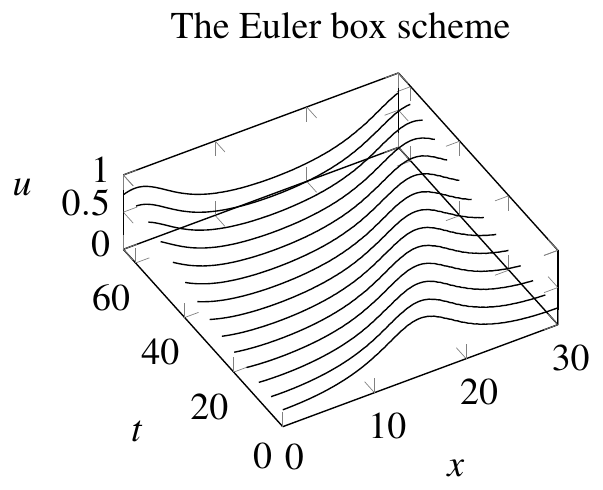}
\includegraphics{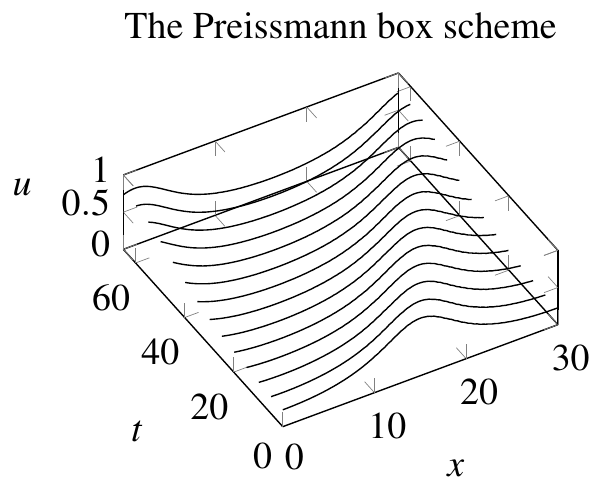};
\includegraphics{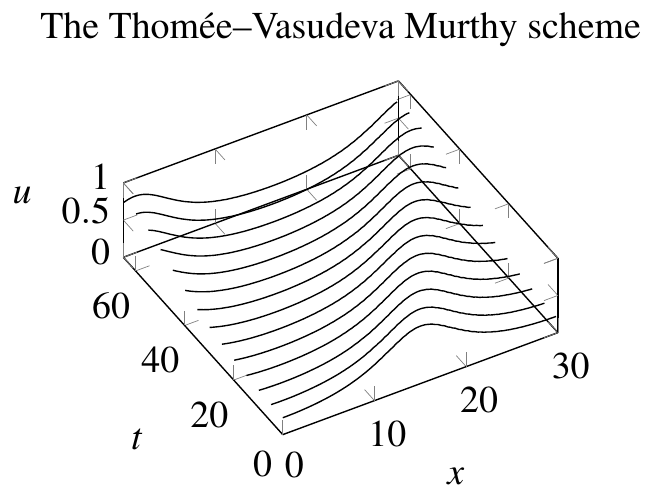}
\includegraphics{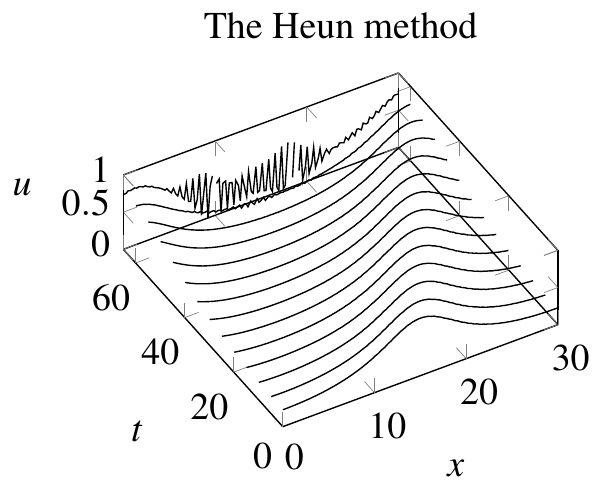}
\caption{Solitary wave solution for the Benjamin--Ono equation.}
\label{fig:BOsoliton}
\end{figure}

\begin{figure}[htbp]
\centering
\includegraphics{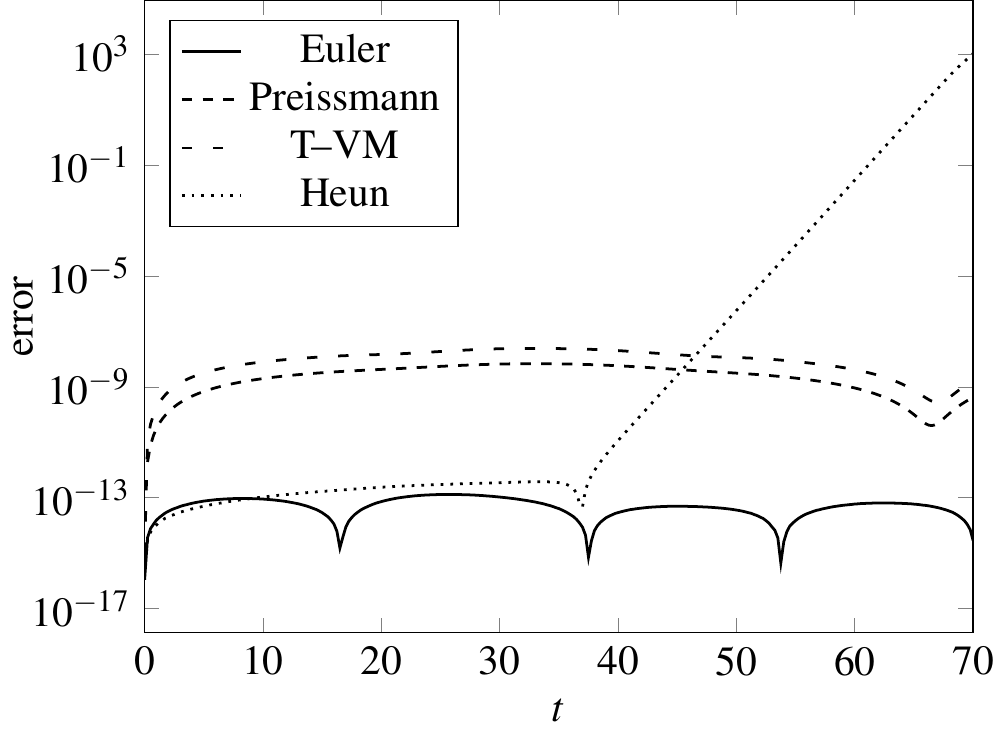}
\caption{Evolution of $\calE$ in the Benjamin--Ono equation.}
\label{fig:BOsolitonEnergy}
\end{figure}

\begin{figure}[htbp]
\centering
\includegraphics{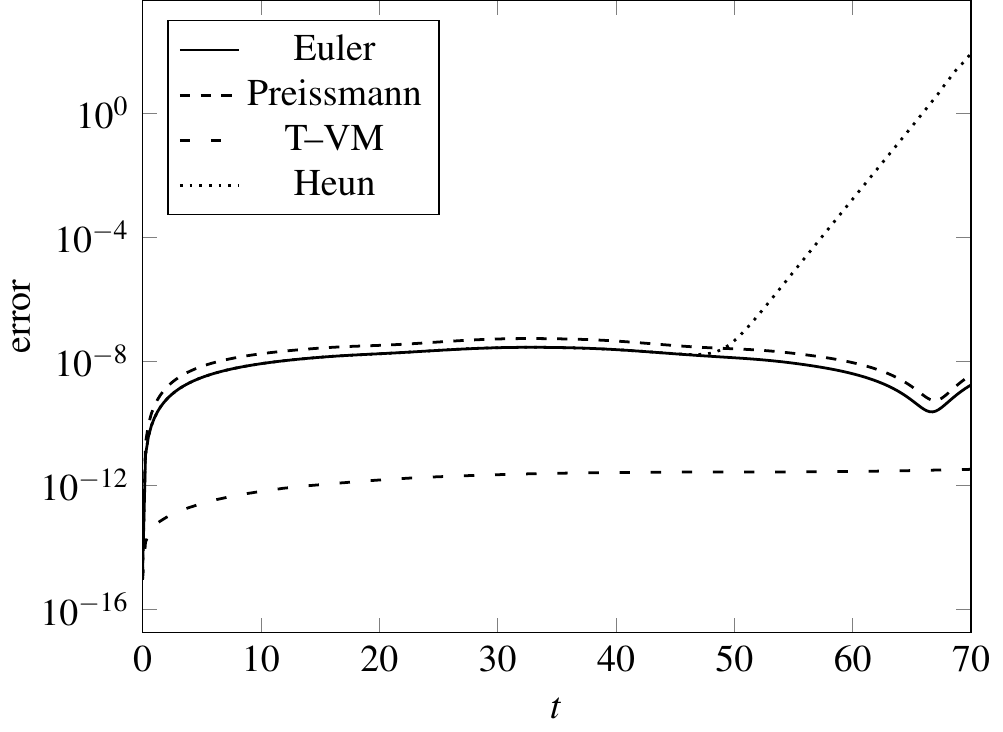}
\caption{Evolution of $\calI$ in the Benjamin--Ono equation.}
\label{fig:BOsolitonMomentum}
\end{figure}

\subsection{Train of solitary waves in the Benjamin equation}

Next, we try the Benjamin equation.
Unfortunately, for this equation, it seems exact solitary wave
solutions have not been discovered in closed form, whereas many efforts
have been done to capture it numerically
(see, for example, \cite{abr99,ddm12,ddm15}. See also~\cite{ap99}.)
Thus here we borrow a numerical setting from~\cite[Section 2.3]{ddm15},
where the break down of a Gaussian packet to solitary wave
solutions is observed.

We take the parameters to $\alpha = -1,\beta = -1, \gamma = 1, \lambda = 1, l = 600$,
$\Delta x = l/2048$, and $\Delta t = 10^{-2}$.
The initial data is chosen to $u(x_n, 0)=2\exp\left(-{(x_n-l/2)^2}/{16}\right)$.
In this experiment, mainly due to the heavy time complexity of the Preissmann box scheme,
we test only the Euler box scheme.

The results are shown in Fig.~\ref{fig:GaussianEB} and Fig.~\ref{fig:GaussianEBinvariants}.
Fig.~\ref{fig:GaussianEB} shows the wave profile at $t=100$.
We observe a train of solitary waves, which is exactly what the preceding study~\cite{ddm15}
observed. 
We here like to note that in~\cite{ddm15} a sophisticated nonlinear scheme was
employed, while in the present paper the Euler box scheme is explicit.
Fig.~\ref{fig:GaussianEBinvariants} shows the evolution of the invariants.
The invariant $\mathcal{E}$ is very well preserved.
The invariant $\mathcal{I}$ deviates from the original value in the early phase
of wave splitting, and then is well preserved after that.

Through this experiment, we conclude that the proposed Euler box scheme
is in fact a good integrator for the Benjamin equation.

\begin{figure}[htbp]
\centering
\includegraphics{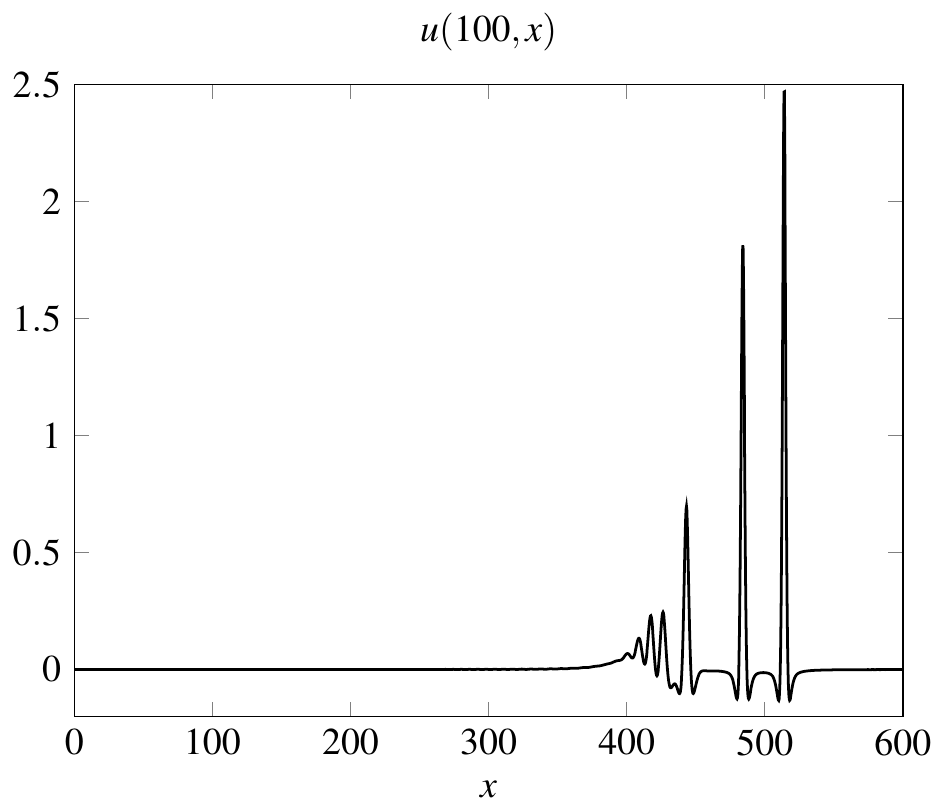}
\caption{Train of solitary waves in the Benjamin equation.}
\label{fig:GaussianEB}
\end{figure}

\begin{figure}[htbp]
\centering
\includegraphics{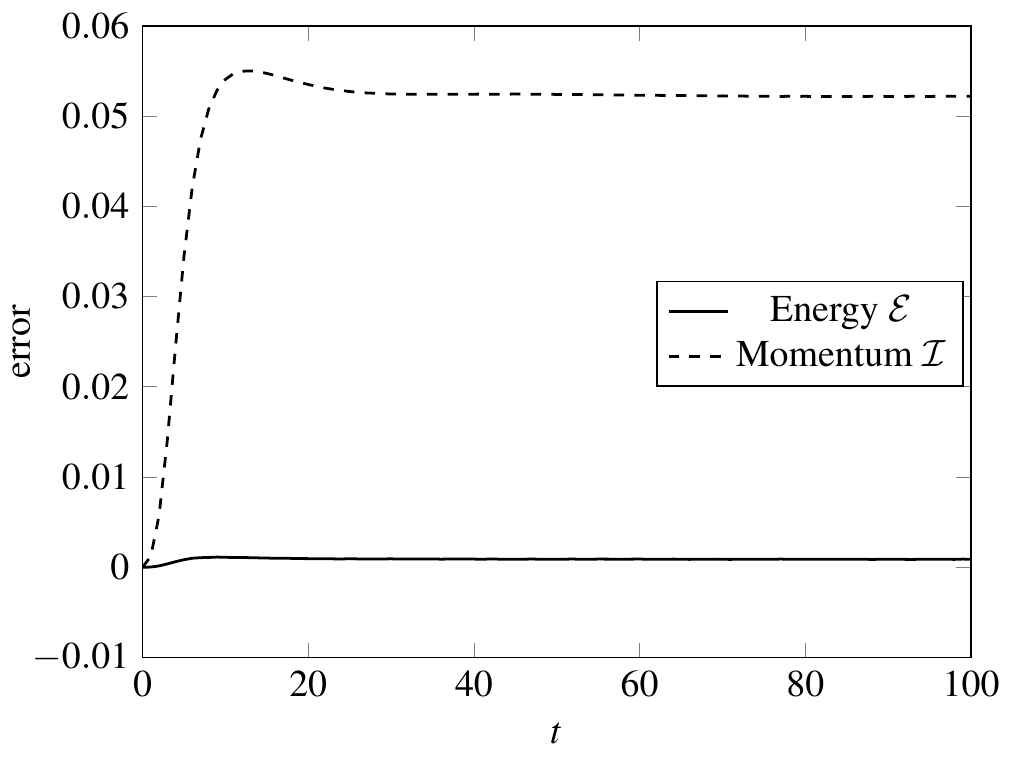}
\caption{Evolution of $\calE$ and $\calI$ in the Euler box scheme for the Benjamin equation.}
\label{fig:GaussianEBinvariants}
\end{figure}

\subsection{Possible wave breaking in the Benjamin equation}

From the experiment above, we regard the proposed Euler box scheme 
is a reliable integrator, at least to a certain extent.
Next, we try to give a new insight about the behavior
of the solutions of the Benjamin equation, utilizing the integrator.
Although in~\cite{li99} the global well-posedness of the equation
has been proved in $L^2(\bbT)$,
that does not necessarily prohibits wave breaking
in other spaces with higher regularity.
The example below might suggest such a possibility.

We take the parameters to $\alpha = 0.01, \beta = 0.001, \gamma = 0.1, \lambda = 0.2, l = 10$,
$\Delta x = l/4096$, and $\Delta t = 10^{-6}$.
The initial data is chosen to
$u(x_n,0) = \cos\left({2\pi x_n}/{l}\right)$.

The results are shown in Fig.~\ref{fig:BlowUpEbu}--\ref{fig:BlowUpInvariants}.
Fig.~\ref{fig:BlowUpEbu} and Fig.~\ref{fig:BlowUpEbux} show the evolution
of $u$ and its derivative $u_x$, respectively.
In Fig.~\ref{fig:BlowUpEbu}, it seems the solution develops a steep slope
around $x \simeq 3$.
Actually in such a place the derivative $u_x$ seems to blow up (Fig.~\ref{fig:BlowUpEbux}).
Fig.~\ref{fig:BlowUpInvariants} shows the evolution of the invariants
$\mathcal{I}$ and $\mathcal{E}$; we see that both are well preserved,
which supports that the result is correct.
Just to confirm this view, we also tried the fourth order Runge--Kutta method
and the space discretization with central differences.
We observed a similar steep slope there (the result omitted here).
We also like to point out that in the Runge--Kutta scheme
the invariants are not
preserved up to the same level as the proposed Euler box scheme
(Fig.~\ref{fig:BlowUpInvariants}).
This means that even with the fourth order temporal accuracy 
this phenomenon is hard to capture, and the structure-preserving
Euler box scheme is much advantageous.

\begin{figure}[htbp]
\centering
\includegraphics{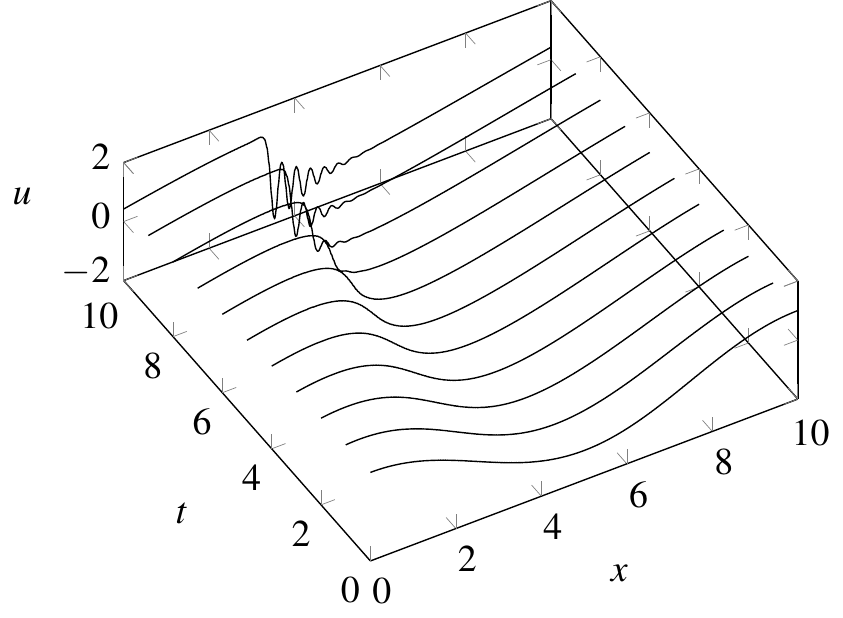}
\caption{The solution captured by the Euler box scheme: Evolution of $u$.}
\label{fig:BlowUpEbu}
\end{figure}

\begin{figure}[htbp]
\centering
\includegraphics{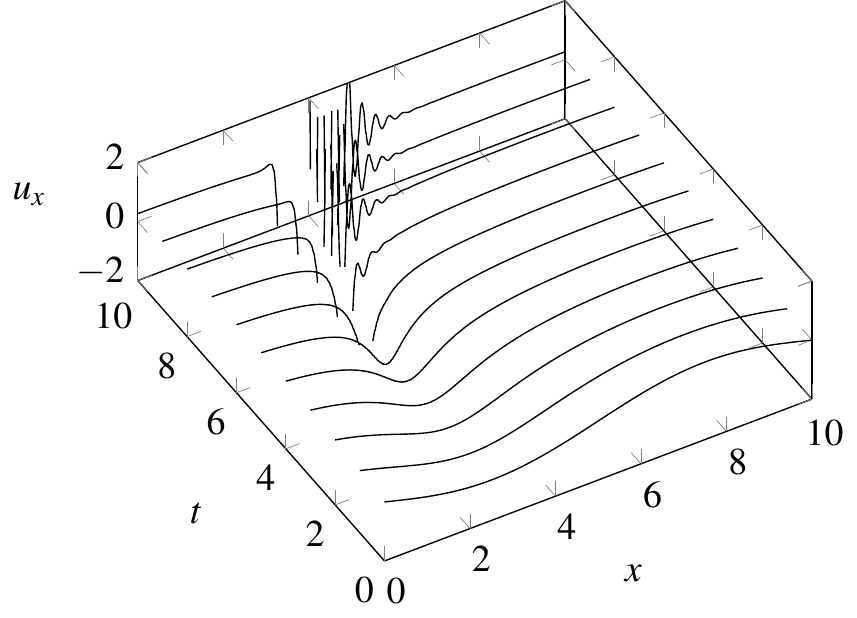}
\caption{The solution captured by the Euler box scheme: Evolution of $u_x$.}
\label{fig:BlowUpEbux}
\end{figure}

\begin{figure}[htbp]
\centering
\includegraphics{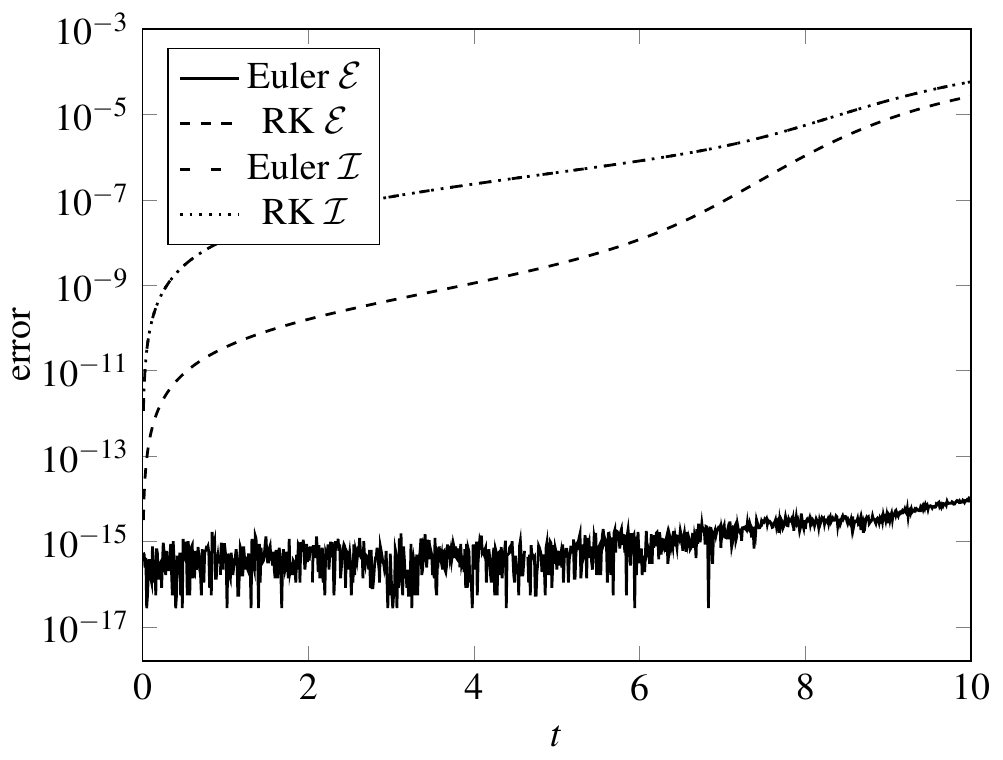}
\caption{The evolution of the invariants.
The results for the momentum $\calI$ are almost identical.}
\label{fig:BlowUpInvariants}
\end{figure}

\section{Concluding remarks}
\label{sec5}

In this paper, we gave a new reformulation of the Benjamin-type equations,
and proposed the Euler box and Preissmann box schemes based on
the formulation.
For the latter scheme, we need the discretization of the Hilbert
transform on the grids with odd number points, and we provided that
with theoretical analysis.
The numerical experiments confirmed the effectiveness of the proposed
structure-preserving schemes.
We hope that these schemes are useful for better understanding
on the behavior of the solutions of the Benjamin-type equations.

One important issue is left unsolved in the present study---%
although we basically followed the line of the discussion
in the standard multi-symplectic method,
the new formulation is not multi-symplectic in the strict sense
of the word.
In fact, as we saw in Section~\ref{sec3},
the standard local conservation laws in usual multi-symplectic PDEs
are lost, and only their weaker versions, regarding global invariants
integrated over space, are allowed to hold.
We have to have a deeper understanding about the geometric 
meaning of the new formulation.

\subsection*{Acknowledgements}
We thank A. Duran who drew our interest to the Benjamin equation.
This work is partly supported by JSPS KAKENHI Grant Numbers
26390126 and 25287030, and also by CREST, JST.

\end{document}